\documentclass{amsart}
\usepackage{amscd}
\usepackage{amsmath,empheq}
\usepackage{amsfonts}
\usepackage{amssymb}
\usepackage{mathrsfs}

\usepackage[all]{xy}

\newtheorem{theorem}{Theorem}

\newtheorem{prop}[theorem]{Proposition}

\newtheorem{definition}[theorem]{Definition}

\newenvironment{proof-sketch}{\noindent{\bf Sketch of Proof}\hspace*{1em}}{\qed\bigskip}

\everymath{\displaystyle}

\newcommand{\RR}{\mathbb R}
\newcommand{\NN}{\mathbb N}

\renewcommand{\leq}{\leqslant}

\renewcommand{\geq}{\geqslant}

\begin{document}

\title[Evolution inclusions with noncoercive viscosity term]{Nonlinear second order evolution inclusions with noncoercive viscosity term}

\author[N.S. Papageorgiou]{Nikolaos S. Papageorgiou}
\address[N.S. Papageorgiou]{National Technical University, Department of Mathematics,
				Zografou Campus, Athens 15780, Greece}
\email{\tt npapg@math.ntua.gr}

\author[V.D. R\u{a}dulescu]{Vicen\c{t}iu D. R\u{a}dulescu}
\address[V.D. R\u{a}dulescu]{Faculty of Applied Mathematics,
AGH University of Science and Technology,
al. Mickiewicza 30, 30-059 Krakow, Poland
 \& Institute of Mathematics ``Simion Stoilow" of the Romanian Academy, P.O. Box 1-764,
          014700 Bucharest, Romania}
\email{\tt vicentiu.radulescu@imar.ro}

\author[D.D. Repov\v{s}]{Du\v{s}an D. Repov\v{s}}
\address[D.D. Repov\v{s}]{Faculty of Education and Faculty of Mathematics and Physics, University of Ljubljana, SI-1000 Ljubljana, Slovenia}
\email{\tt dusan.repovs@guest.arnes.si}


\begin{abstract}
In this paper we deal with a second order nonlinear evolution inclusion, with a nonmonotone, noncoercive viscosity term. Using a parabolic regularization (approximation) of the problem and {\it a priori} bounds that permit passing to the limit, we prove that the problem has a solution.
\end{abstract}

\maketitle


\section{Introduction}\label{sec1}

Let $T=[0,b]$ and let $(X, H, X^*)$ be an evolution triple of spaces, with the embedding of $X$ into $H$ being compact (see Section \ref{sec2} for definitions).

In this paper, we study the following nonlinear evolution inclusion:
\begin{equation}\label{eq1}
	\left\{
		\begin{array}{ll}
			u''(t) + A(t, u'(t)) + Bu(t) \in F(t, u(t), u'(t))\ \mbox{for almost all}\ t\in T,\\
			u(0) = u_{0},\ u'(0) = u_1.
		\end{array}
	\right\}
\end{equation}

In the past, such multi-valued problems were studied by Gasinski \cite{3}, Gasinski and Smolka \cite{6, 7}, Mig\'orski {\it et al.} \cite{migo1, 11, 12, migo2}, Ochal \cite{13}, Papageorgiou, R\u adulescu and Repov\v{s} \cite{prr1,prr2}, Papageorgiou and Yannakakis \cite{14, 15}. The works of Gasinski \cite{3}, Gasinski and Smolka \cite{6, 7} and Ochal \cite{13}, all deal with hemivariational inequalities, that is, $F(t, x, y)=\partial J(x)$ with $J(\cdot)$ being a locally Lipschitz functional and $\partial J(\cdot)$ denoting the Clarke subdifferential of $J(\cdot)$. In Papageorgiou and Yannakakis \cite{14, 15},  the multivalued term $F(t, x, y)$ is general (not necessarily of the subdifferential type) and depends also on the time derivative of the unknown function $u(\cdot)$. With the exception of Gasinski and Smolka \cite{7}, in all the other works the viscosity term $A(t, \cdot)$ is assumed to be coercive or zero. In the work of Gasinski and Smolka \cite{7}, the viscosity term is autonomous (that is, time independent) and $A:X\rightarrow X^*$ is linear and bounded.

In this work, the viscosity term $A:T\times X\rightarrow X^*$ is time dependent, noncoercive, nonlinear and nonmonotone in $x\in X$. In this way, we extend and improve the result of Gasinski and Smolka \cite{7}. Our approach uses a kind of parabolic regularization of the inclusion, analogous to the one used by Lions \cite[p. 346]{10} in the context of semilinear hyperbolic equations.

\section{Mathematical Background and Hypotheses}\label{sec2}

Let $V, Y$ be Banach spaces and assume that $V$ is embedded continuously and densely into $Y$ (denoted by $V\hookrightarrow Y$). Then we have the following properties:\\
\indent	(i) $Y^*$ is embedded continuously into $V^*$;\\
\indent	(ii) if $V$ is reflexive, then $Y^*\hookrightarrow V^*$.

The following notion is a useful tool in the theory of evolution equations.
\begin{definition}\label{def1}
	By an ``evolution triple" (or ``Gelfand triple") we understand a triple of spaces $(X, H, X^*)$ such that
	\begin{itemize}
		\item [(a)] $X$ is a separable reflexive Banach space and $X^*$ is its topological dual;
		\item [(b)] $H$ is a separable Hilbert space identified with its dual $H^*$, that is, $H=H^*$ (pivot space);
		\item [(c)] $X\hookrightarrow H$.
	\end{itemize}
\end{definition}
Then from the initial remarks we have
$$X\hookrightarrow H = H^*\hookrightarrow X^*.$$

In what follows, we denote by $||\cdot||$ the norm of $X$, by $|\cdot|$ the norm of $H$ and by $||\cdot||_*$ the norm of $X^*$. Evidently we can find $\hat{c}_1, \hat{c}_2 > 0$ such that
$$|\cdot|\leqslant\hat{c}_1||\cdot||\ \mbox{and}\ ||\cdot||_*\leqslant\hat{c}_2|\cdot|\,.$$
By $(\cdot, \cdot)$ we denote the inner product of $H$ and by $\langle\cdot, \cdot\rangle$ the duality brackets for the pair $(X^*, X)$. We have
\begin{equation}\label{eq2}
	\langle\cdot, \cdot\rangle|_{H\times X} = (\cdot, \cdot).
\end{equation}

Let $1<p<\infty$. The following space is important in the study of problem \eqref{eq1}:
$$W_p(0, b)=\left\{u\in L^p(T, X): u'\in L^{p'}(T, X^*)\right\}\ \left(\frac{1}{p} + \frac{1}{p'} = 1\right).$$
Here $u'$ is understood in the distributional sense (weak derivative). We know that $L^p(T, X)^*$ $= L^{p'}(T, X^*)$ (see, for example, Gasinski and Papageorgiou \cite[p. 129]{4}). Suppose that $u\in W_p(0,b)$. If we view $u(\cdot)$ as an $X^*$-valued function, then $u(\cdot)$ is absolutely continuous, hence differentiable almost everywhere and this derivative coincides with the distributional one. So, $u'\in L^{p'}(T, X^*)$ and we can say
$$W_p(0, b) \subseteq AC^{1,p'}(T, X^*) = W^{1,p'}((0, b), X^*).$$
The space $W_p(0, b)$ is equipped with the norm
$$||u||_{W_p}=\left[||u||^p_{L^p(T, X)} + ||u'||^p_{L^{p'}(T, X^*)}\right]^{\frac{1}{p}}\ \mbox{for all}\ u\in W_p(0,b).$$
Evidently, another equivalent norm on $W_p(0,b)$ is
$$|u|_{W_p}=||u||_{L^p(T, X)} + ||u'||_{L^p(T, X^*)}\ \mbox{for all}\ u\in W_p(0,b).$$
With any of the above norms, $W_p(0,b)$ becomes a separable reflexive Banach space. We have that
\begin{eqnarray}
	& & W_p(0, b) \hookrightarrow C(T, H); \label{eq3} \\
	& & W_p(0, b) \hookrightarrow L^p(T, H)\ \mbox{and the embedding is compact}. \label{eq4}
\end{eqnarray}

The elements of $W_p(0, b)$ satisfy an integration by parts formula which will be useful in our analysis.

\begin{prop}\label{prop2}
	If $u, v\in W_p(0, b)$ and $\xi(t)=(u(t), v(t))$ for all $t\in T$, then $\xi(\cdot)$ is absolutely continuous and $\frac{d\xi}{dt}(t) = \langle u'(t), v(t)\rangle + \langle u(t), v'(t)\rangle$ for almost all $t\in T$.
\end{prop}

Now suppose that $(\Omega, \Sigma, \mu)$ is a finite measure space, $\Sigma$ is $\mu-complete$ and $Y$ is a separable Banach space. A multifunction (set-valued function) $F:\Omega\rightarrow 2^Y\backslash\{\emptyset\}$ is said to be ``graph measurable", if
$${\rm Gr}\, F=\{(\omega, y)\in\Omega\times Y:y\in F(\omega)\}\in\Sigma\times B(Y),$$
with $B(Y)$ being the Borel $\sigma$-field of $Y$.

If $F(\cdot)$ has closed values, then graph measurability is equivalent to saying that for every $y\in Y$ the $\RR_+$-valued function $$\omega\mapsto d(y, F(\omega))=\inf\{||y-v||_Y:v\in F(\omega)\}$$ is $\Sigma$-measurable.

Given a graph measurable multifunction $F:\Omega\rightarrow2^Y\backslash\{\emptyset\}$, the Yankov-von Neumann-Aumann selection theorem (see Hu and Papageorgiou \cite[p. 158]{8}) implies that $F(\cdot)$ admits a measurable selection, i.e. that there exists $f:\Omega\rightarrow Y$ a $\Sigma$-measurable function such that $f(\omega)\in F(\omega)$ $\mu$-almost everywhere. In fact, we can find an entire sequence $\{f_n\}_{n\geqslant1}$ of measurable selections such that $F(\omega)\subseteq\overline{\{f_n(\omega)\}}_{n\geq1}$ $\mu$-almost everywhere.

For $1\leq p\leq\infty$, we define
$$S^p_F = \{f\in L^p(\Omega, Y): f(\omega)\in F(\omega)\ \mu\mbox{-almost everywhere}\}.$$
It is easy to see that $S^p_F\neq\emptyset$ if and only if $\omega\mapsto\inf\{||v||_Y:v\in F(\omega)\}$ belongs to $L^p(\Omega)$. This set is ``decomposable" in the sense that if $(A, f_1, f_2)\in\Sigma\times S^p_F\times S^p_F$, then
$$\chi_A f_1 + \chi_{A^c} f_2 \in S^p_F.$$
Finally, for a sequence $\{C_n\}_{n\geq1}$ of nonempty subsets of $Y$, we define
$$w-\limsup_{n\rightarrow\infty}C_n = \{y\in Y: y= w-\lim_{k\rightarrow\infty}y_{n_k}, y_{n_k}\in C_{n_k}, n_1<n_2<\cdots <n_k<\cdots\}.$$
For more details on the notions discussed in this section, we refer to Gasinski and Papageorgiou \cite{4}, Roubi\v{c}ek \cite{16}, Zeidler \cite{17} (for evolution triples and related notations) and Hu and Papageorgiou \cite{8} (for measurable multifunctions).

Let $V$ be a reflexive Banach space and $A:V\rightarrow V^*$ a map. We say that $A$ is ``pseudomonotone", if $A$ is continuous from every finite dimensional subspace of $V$ into $V^*_w$ (= the dual $V^*$ equipped with the weak topology) and if
$$v_n\xrightarrow{w}v\ \mbox{in}\ V,\ \limsup_{n\rightarrow\infty}\langle A(v_n), v_n-v\rangle\leq0$$
then
$$\langle A(v), v-y\rangle\leq\liminf_{n\rightarrow\infty}\langle A(v_n), v_n-y\rangle\ \mbox{for all}\ y\in V.$$

An everywhere defined maximal monotone operator is pseudomonotone. If $V$ is finite dimensional, then every continuous map $A:V\rightarrow V^*$ is pseudomonotone.

In what follows, for any Banach space $Z$, we will use the following notations:
\begin{eqnarray*}
	& & P_{f(c)}(Z) = \{C\subseteq Z: C \ \mbox{is nonempty, closed (and convex)}\}, \\
	& & P_{(w)k(c)}(Z) = \{C\subseteq Z: C\ \mbox{is nonempty, (weakly-) compact (and convex)}\}.
\end{eqnarray*}
The hypotheses on the data of problem (\ref{def1}) are the following:

$H(A):$ $A:T\times T\rightarrow X^*$ is a map such that

\smallskip
\begin{itemize}
	\item [(i)] for all $y\in X, t\mapsto A(t, y)$ is measurable;
	\item [(ii)] for almost all $t\in T$, the map $y\mapsto A(t, y)$ is pseudomonotone;
	\item [(iii)] $||A(t, y)||_*\leq a_1(t) + c_1||y||^{p-1}$ for almost all $t\in T$ and all $y\in X$, with $a_1\in L^{p'}(T)$, $c_1>0$, $2\leq p<\infty$;
	\item [(iv)] $\langle A(t, y), y\rangle \geq 0$ for almost all $t\in T$ and all $y\in X$.
\end{itemize}

\smallskip
$H(B):$ $B\in \mathscr{L}(X, X^*)$, $\langle Bx, y\rangle = \langle x, By\rangle$ for all $x, y\in X$ and $\langle Bx, x\rangle\geq c_0||x||^2$ for all $x\in X$ and some $c_0>0$.

\smallskip
$H(F):$ $F:T\times H\times H\rightarrow P_{f_c}(H)$ is a multifunction such that
\begin{itemize}
	\item [(i)] for all $x, y\in H$, $t\mapsto F(t, x, y)$ is graph measurable;
	\item [(ii)] for almost all $t\in T$, the graph ${\rm Gr}\, F(t, \cdot, \cdot)$ is sequentially closed in $H\times H_w\times H_w$ (here $H_w$ denotes the Hilbert space $H$ furnished with the weak topology);
	\item [(iii)] $|F(t, x, y)|=\sup\{|h|:h\in F(t, x, y)\}\leq a_2(t)(1+|x|+|y|)$ for almost all $t\in T$ and all $x, y\in H$ with $a_2\in L^2(T)_+$.
\end{itemize}
\begin{definition}\label{def3}
	We say that $u\in C(T, X)$ is a ``solution" of problem (\ref{eq1}) with $u_0\in X,\ u_1\in H$, if
	\begin{itemize}
		\item $u'\in W_p(0, b)$ and
		\item there exists $f\in S^2_{F(\cdot, u(\cdot), u'(\cdot))}$ such that
		\begin{equation*}
			\left\{
				\begin{array}{ll}
					u''(t) + A(t, u'(t)) + Bu(t) = f(t)\ \mbox{for almost all}\ t\in T,\\
					u(0) = u_0, u'(0)=u_1.
				\end{array}
			\right\}	
		\end{equation*}
	\end{itemize}
\end{definition}

In what follows, we denote by $S(u_0, u_1)$ the set of solutions of problem (1). Recalling that $W_p(0, b)\hookrightarrow C(T, H)$ (see (\ref{eq3})), we have that
$$S(u_0, u_1)\subseteq C^1(T, H).$$
By Troyanski's renorming theorem (see Gasinski and Papageorgiou \cite[p. 911]{4}) we may assume without  loss of generality  that both $X$ and $X^*$ are locally uniformly convex.
Let $\mathcal{F}:X\rightarrow X^*$ be the duality map of $X$ defined by
$$\mathcal{F}(x)=\{x^*\in X^*: \langle x^*,x\rangle=||x||^2=||x^*||^2_*\}.$$
We know that $\mathcal{F}(\cdot)$ is single-valued and a homeomorphism (see Gasinski and Papageorgiou \cite[p. 316]{4} and Zeidler \cite[p. 861]{17}).

For every $r\geq p$, let $K_r:X\rightarrow X^*$ be the map defined by
$$K_r(y)=||y||^{r-2} \mathcal{F}(y)\ \mbox{for all}\ y\in X.$$

\section{Existence Theorem}

Given $\epsilon>0$, we consider the following perturbation (parabolic regularization) of problem (\ref{eq1}):
\begin{equation}\label{eq5}
	\left\{
		\begin{array}{ll}
			u''(t) + A(t, u'(t)) + \epsilon K_r(u'(t)) + Bu(t) \in F(t, u(t), u'(t))\ \mbox{for a.a.}\ t\in T,\\
			u(0) = u_0,\ u'(0)=u_1.
		\end{array}
	\right\}
\end{equation}

Consider the map $A_\epsilon:T\times X\rightarrow X^*$ defined by
$$A_\epsilon(t, y) = A(t, y) + \epsilon K_r(y)\ \mbox{for all}\ t\in T,\ \mbox{and all}\ y\in X.$$
This map has the following properties:
\begin{itemize}
	\item [(i)] for all $y\in X$, the map $t\mapsto A_\epsilon(t, y)$ is measurable;
	\item [(ii)] for almost all $t\in T$, the map $y\mapsto A_\epsilon(t, y)$ is pseudomonotone;
	\item [(iii)] $||A_\epsilon(t, y)||_*\leq\hat{a}_1(t)+\hat{c}_1||y||^{r-1}$ for almost all $t\in T$, all $y\in X$ and with $\hat{a}_1\in L^{p'}(T)$, $\hat{c}_1>0$ (recall that $r\geq p$ and $\frac{1}{r}+\frac{1}{r'}=1$);
	\item [(iv)] $\langle A_\epsilon(t, y), y\rangle\geq\epsilon||y||^r$ for all $t\in T$, all $y\in X$.
\end{itemize}

So, in problem \eqref{eq1} the viscosity term $A_\epsilon(t, \cdot)$ is coercive. Therefore we can apply Theorem 1 of Papageorgiou and Yannakakis \cite{14} and we obtain the following existence result for the approximate (regularized) problem (\ref{eq5}).
\begin{prop}\label{prop4}
	If hypotheses $H(A),\ H(B),\ H(F)$ hold and $u_0\in X, u_1\in H$, then problem (\ref{eq5}) admits a solution $u_\epsilon\in W^{1,r}((0, b), X)\cap C^1(T, H)$ with
	$$u'_\epsilon\in W_r(0, b).$$
\end{prop}

To produce a solution for the original problem (\ref{eq1}), we have to pass to the limit as $\epsilon\rightarrow0^+$. To do this, we need to have {\it a priori} bounds for the solutions $u_\epsilon(\cdot)$ which are independent of $\epsilon\in(0, 1]$ and $r\geq p$.
\begin{prop}\label{prop5}
	If hypotheses $H(A), H(B), H(F)$ hold, $u_0\in X, u_1\in H$ and $u(\cdot)$ is a solution of (\ref{eq5}), then there exists $M_0>0$ which is independent of $\epsilon\in(0, 1]$ and $r\geq p$ for which we have
	$$||u||_{C(T, X)},\ ||u'||_{C(T, H)},\ \epsilon^{\frac{1}{r}}||u'||_{L^r(T, X)},\ ||u''||_{L^2(T, X^*)}\leq M_0.$$
\end{prop}
\begin{proof}
	It follows from Proposition \ref{prop4} that $u'\in W_r(0, b)$ and that there exists $f\in S^2_{F(\cdot, u(\cdot), u'(\cdot))}$ such that
	$$u''(t) + A(t, u'(t)) + \epsilon K_r(u'(t)) + Bu(t) = f(t)\ \mbox{for almost all}\ t\in T.$$
	We act with $u'(t)\in X$. Then
	\begin{eqnarray}\label{eq6}
		&&\langle u''(t), u'(t)\rangle + \langle A(t, u'(t)), u'(t)\rangle + \epsilon\langle K_r(u'(t)), u'(t)\rangle = (f(t), u'(t))\\
		&& \mbox{for almost all}\ t\in T\ \mbox{(see (\ref{eq2}))}.\nonumber
	\end{eqnarray}

	We examine separately each summand on the left-hand side of (\ref{eq6}). Recall that $u'_r\in W_r(0, b)$. So from Proposition \ref{prop2} (the integration by parts formula), we have
	\begin{equation}\label{eq7}
		\langle u''(t), u'(t)\rangle = \frac{1}{2}\frac{d}{dt}|u'(t)|^2\ \mbox{for almost all}\ t\in T.
	\end{equation}
	Hypothesis $H(A)(iv)$ and the definition of the duality map, imply that
	\begin{equation}\label{eq8}
		\langle A(t, u'(t)), u'(t)\rangle + \epsilon\langle K_r(u'(t)), u'(t)\rangle \geq \epsilon||u'(t)||^r\ \mbox{for almost all}\ t\in T.
	\end{equation}
	By hypothesis $H(B)$, we have
	\begin{equation}\label{eq9}
		\langle Bu(t), u'(t)\rangle = \frac{1}{2}\frac{d}{dt} \langle Bu(t), u(t)\rangle\
\mbox{for almost all}\ t\in T.
	\end{equation}
	We return to (\ref{eq6}) and use (\ref{eq7}), (\ref{eq8}), (\ref{eq9}). We obtain
	\begin{eqnarray}
		& & \frac{1}{2}\frac{d}{dt}|u'(t)|_2 + \epsilon||u'(t)||^r + \frac{1}{2}\frac{d}{dt}\langle Bu(t), u(t)\rangle \leq(f(t), u'(t))\ \mbox{for a.a.}\ t\in T, \nonumber \\
		& \Rightarrow & \frac{1}{2}|u'(t)|^2 + \epsilon\int_0^t||u'(s)||^r ds + c_0||u(t)||^2 \nonumber \\
		& \leq & \int_0^t(f(s), u'(s))ds + \frac{1}{2}|u_1|^2 + \frac{1}{2}||B||_\mathscr{L}||u_0||^2\ \mbox{(see hypothesis $H(B)$)}. \label{eq10}
	\end{eqnarray}
	Using hypothesis $H(F)(iii)$, we get
	\begin{eqnarray}
		& & \int_0^t(f(s), u'(s))ds \nonumber \\
		& \leq & \int_0^t\left[a_2(s) + a_2(s)\left(|u(s)| + |u'(s)|\right)\right]|u'(s)|ds \nonumber \\
		& \leq & \int_0^t|u'(s)|^2 ds + \int_0^t a_2(s)^2 ds + \int_0^t a_2(s)^2\left[|u(s)|^2 + |u'(s)|^2\right]ds . \label{eq11}
	\end{eqnarray}

	Recall that $u\in W^{1,r}((0, b), X)$ (see Proposition \ref{prop4}). So, $u\in AC^{1, r}(T, H)$ and we can write
	\begin{eqnarray}
		& & u(t) = u_0 + \int_0^t u'(s)ds\ \mbox{for all}\ t\in T \nonumber \\
		& \Rightarrow & |u(t)|^2 \leq 2|u_0|^2 + 2b\int_0^t|u'(s)|^2 ds\ \mbox{for all}\ t\in T\ \mbox{(using Jensen's inequality)}. \label{eq12}
	\end{eqnarray}
	We use (\ref{eq12}) in (\ref{eq11}) and obtain
	\begin{eqnarray}
		& & \int_0^t(f(s), u'(s))ds \nonumber \\
		& \leq & ||a_2||_2^2 + \int_0^t\left[1+a_2(s)^2\right]|u'(s)|^2ds + \int_0^t2a_2(s)^2\left[|u_0|^2 + b\int_0^s|u'(\tau)|^2d\tau\right]ds \nonumber \\
		& \leq &  c_2 + \int_0^t\eta(s)|u'(s)|^2ds + 2b\int_0^ta_2(s)^2\int_0^s|u'(\tau)|^2d\tau ds\label{eq13}\\
		& & \mbox{for some $c_2>0$ and $\eta\in L^{1}(T)$}.  \nonumber
	\end{eqnarray}
	We use (\ref{eq13}) in (\ref{eq10}) and have
	\begin{eqnarray}
		& & \frac{1}{2}|u'(t)|^2 + \epsilon\int_0^t||u'(s)||^p ds + c_0||u(t)||^2 \nonumber \\
		& \leq & c_3 + \int_0^t\eta(s)|u'(s)|^2 ds + 2b\int_0^t a_2(s)^2\int_0^s|u'(\tau)|^2d\tau ds\ \mbox{for some $c_3>0$}. \label{eq14}
	\end{eqnarray}
	Invoking Proposition 1.7.87 of Denkowski, Mig\'orski and Papageorgiou \cite[p. 128]{2} we can find $M>0$ (independent of $\epsilon\in(0, 1]$ and $r\geq p$) such that
	\begin{eqnarray*}
		& & |u'(t)|^2\leq M\ \mbox{for all}\ t\in T, \\
		& \Rightarrow & ||u'||_{C(T, H)} \leq M_1 = M^{\frac{1}{2}}.
	\end{eqnarray*}
	Using this bound in (\ref{eq14}), we can find $M_2>0$ (independent of $\epsilon\in(0, 1]$ and $r\geq p$) such that
	$$||u||_{C(T, X)}\leq M_2\ \mbox{and}\ \epsilon^{\frac{1}{r}}||u'||_{L^r(T, X)} \leq M_2.$$

	Finally, directly from (\ref{eq5}), we see that there exists $M_3>0$ (independent of $\epsilon\in(0, 1]$ and $r\geq p$) such that
	$$||u''||_{L^{r'}}(T, X^*) \leq M_3.$$
	We set $M_0=\max\{M_1, M_2, M_3\}>0$ and get the desired bound.
\end{proof}

The bounds produced in Proposition \ref{prop5} permit passing to the limit as $\epsilon\rightarrow0^+$ to produce a solution for problem (\ref{eq1}).
\begin{theorem}\label{th6}
	If hypotheses $H(A), H(B), H(F)$ hold and $u_0\in X, u_1\in H$, then $S(u_0, u_1)\neq\emptyset$.
\end{theorem}
\begin{proof}
	Let $\epsilon_n\rightarrow0^+$ and let $u_n=u_{\epsilon_n}$ be solutions of the ``regularized" problem (\ref{eq5}) (see Proposition \ref{prop4}). Because of the bounds established in Proposition \ref{prop5} and by passing to a suitable subsequence if necessary, we can say that
	\begin{equation}\label{eq15}
		\left\{
			\begin{array}{ll}
				u_n\xrightarrow{w^*}u\ \mbox{in}\ L^\infty(T, X),\ u_n\xrightarrow{w}u\ \mbox{in}\ C(T, H),\ u_n\rightarrow u\ \mbox{in}\ L^r(T, H) \\
				u'_n\xrightarrow{w^*}y\ \mbox{in}\ L^\infty(T, H),\ u''_n\xrightarrow{w}v\ \mbox{in}\ L^{r'}(T, X^*)\ \mbox{(see (\ref{eq3}) and (\ref{eq4}))}.
			\end{array}
		\right\}
	\end{equation}
	Recall that $u_n\in AC^{1,r}(T, H)$ for all $n\in \NN$ and so
	\begin{eqnarray*}
		& & u_n(t) = u_0 + \int_0^t u'_n(s)ds\ \mbox{for all}\ t\in T, \\
		& \Rightarrow & u(t) = u_0 + \int_0^t y(s)ds\ \mbox{for all}\ t\in T\ \mbox{(see (\ref{eq15}))}, \\
		& \Rightarrow & u\in AC^{1,r}(T, H)\ \mbox{and}\ u'=y.
	\end{eqnarray*}

	Since $u_n\in W_r(0, b)$ for all $n\in\NN$, we have
	$$v=y'=u''\in L^{r'}(T, X^*)\ \mbox{(see Hu and Papageorgiou \cite[p. 6]{9})}.$$
	
Let $a:L^r(T, X)\rightarrow L^{r'}(T, X^*)$ be the nonlinear map defined by
	$$a(u)(\cdot) = A(\cdot, u(\cdot))\ \mbox{for all}\ u\in L^r(T, X).$$
	Also, let $\hat{K}_r:L^r(T, X)\rightarrow L^{r'}(T, X^*)$ be defined by
	$$\hat{K}_r(u)(\cdot) = ||u(\cdot)||^{r-2} \mathscr{F}(u(\cdot))\ \mbox{for all}\ u\in L^r(T, X).$$
	Both maps are continuous and monotone, hence maximal monotone (see Gasinski and Papageorgiou \cite[Corollary 3.2.32, p. 320]{4}).
	
	Finally, let $\hat{B}\in\mathscr{L}(L^r(T, X), L^{r'}(T, X^*))$ be defined by
	$$\hat{B}(u)(\cdot) = B(u(\cdot))\ \mbox{for all}\ u\in L^r(T, X).$$
	We have
	\begin{eqnarray}
		u''_n + a(u'_n) + \epsilon_n\hat{K}_r(u'_n) + \hat{B}u_n = f_n\ \mbox{in}\ L^r(T,X^*) \label{eq16} \\
		\mbox{with}\ f_n\in S^2_{F(\cdot, u_n(\cdot), u'_n(\cdot))}\ \mbox{for all}\ n\in\NN. \nonumber
	\end{eqnarray}
	From (\ref{eq15}) we have
	\begin{eqnarray}
		& & u_n\xrightarrow{w}u\ \mbox{in}\ L^r(T, X), \nonumber \\
		& \Rightarrow & \hat{B} u_n\xrightarrow{w}\hat{B}u\ \mbox{in}\ L^{r'}(T, X^*)\ \mbox{as}\ n\rightarrow\infty . \label{eq17}
	\end{eqnarray}
	Also, we have
	\begin{eqnarray}\label{eq18}
		& & ||\hat{K}_r(u'_n)||_{L^{r'}(T, X^*)} = ||u'_n||^{r-1}_{L^r(T,X)}, \nonumber \\
		& \Rightarrow & \epsilon_n||\hat{K}_r(u'_n)||_{L^{r'}(T, X^*)} = \epsilon_n^{\frac{1}{r}}\left(\epsilon_n^{\frac{1}{r}}||u'_n||_{L^r(T, X)}\right)^{r-1}\ \mbox{(recall that $\frac{1}{r} + \frac{1}{r'} = 1$)} \nonumber \\
		& & \leq \epsilon_n^{\frac{1}{r}}M_0^{r-1}\ \mbox{for all}\ n\in\NN\ \mbox{(see Proposition \ref{prop5})} \nonumber \\
		& \Rightarrow & \epsilon_n||\hat{K}_r(u'_r)||_{L^{r'}(T, X^*)}\rightarrow0\ \mbox{as}\ n\rightarrow\infty
	\end{eqnarray}
	From (\ref{eq15}) and since $v=u''$, we have
	\begin{equation}\label{eq19}
		u''_n\xrightarrow{w}u''\ \mbox{in}\ L^{r'}(T, X^*).
	\end{equation}
	Finally, hypothesis $H(F)(iii)$ and Proposition \ref{prop5} imply that
	$$\{f_n\}_{n\geq1}\subseteq L^{2}(T, H)\ \mbox{is bounded}.$$
	By passing to a subsequence if necessary, we may assume that
	$$f_n\xrightarrow{w}f\ \mbox{in}\ L^{2}(T, H).$$
	Invoking Proposition 3.9 of Hu and Papageorgiou \cite[p. 694]{8}, we have
	\begin{eqnarray}
		& f(t) & \in \overline{\rm conv}\, w-\limsup_{n\rightarrow\infty}\{f_n(t)\} \nonumber \\
		& & \leq \overline{\rm conv}\, w-\limsup_{n\rightarrow\infty} F(t, u_n(t), u'_n(t))\ \mbox{for almost all}\ t\in T\ \mbox{(see (\ref{eq16}))}. \label{eq20}
	\end{eqnarray}
	From (\ref{eq15}) we see that
	$$u'_n\xrightarrow{w}u'\ \mbox{in}\ W^{1, r'}((0, b), X^*).$$
	Recall that $W^{1, r'}((0, b), X^*)\hookrightarrow C(T, X^*)$. So, it follows that
	\begin{eqnarray}
		& & u_n'\xrightarrow{w}u'\ \mbox{in}\ C(T, X^*) \nonumber \\
		& \Rightarrow & u_n'(t)\xrightarrow{w}u'(t)\ \mbox{in}\ X^*\ \mbox{for all}\ t\in T .\label{eq21}
	\end{eqnarray}
	On the other hand, by Proposition \ref{prop5} we have
	$$|u'_n(t)|\leq M_0\ \mbox{for all}\ t\in T,\ \mbox{all}\ n\in\NN.$$
	So, by passing to a subsequence ({\it a priori} the subsequence depends on $t\in T$), we have
	\begin{eqnarray*}
		& & u_n'(t)\xrightarrow{w}\hat{y}(t)\ \mbox{in}\ H \\
		& \Rightarrow & \hat{y}(t) = u'(t)\ \mbox{for all}\ t\in T\ \mbox{(see (\ref{eq21}))}.
	\end{eqnarray*}
	Hence for the original sequence we have
	\begin{equation}\label{eq22}
		u'_n(t)\xrightarrow{w}u'(t)\ \mbox{in}\ H\ \mbox{for all}\ t\in T.
	\end{equation}

	We know that $\{u_n\}_{n\geq1}\subseteq W_{r}(0, b)$ is bounded (see Proposition \ref{prop5}) and recall that $W_r(0, b)\hookrightarrow L^{r}(T, H)$ compactly (see (\ref{eq4})). From this compact embedding and from (\ref{eq22}), we obtain
	\begin{equation}\label{eq23}
		u_n(t)\rightarrow u(t)\ \mbox{in}\ H\ \mbox{for all}\ t\in T\ \mbox{as}\ n\rightarrow\infty.
	\end{equation}
	From (\ref{eq20}), (\ref{eq22}), (\ref{eq23}) and hypothesis $H(F)(iii)$ we infer that
	\begin{eqnarray*}
		& & f(t)\in F(t, u(t), u'(t))\ \mbox{for almost all}\ t\in T, \\
		& \Rightarrow & f\in S^2_{F(\cdot, u(\cdot), u'(\cdot))}.
	\end{eqnarray*}

	In what follows, we denote by $((\cdot, \cdot))$ the duality brackets for the pair $$(L^r(T, X^*), L^r(T, X)).$$ Acting with $u'_n-u'\in L^r(T, X)$ on (\ref{eq16}), we have
	\begin{eqnarray}
		& & ((u''_n, u'_n-u{'})) + ((a(u'_n), u'_n-u')) + ((\epsilon_n\hat{K}_r(u'_n), u'_r-u')) + ((\hat{B}u_n, u'_n-u')) \nonumber \\
		& & = \int_0^b(f_n, u'_n-u')dt\ \mbox{for all}\ n\in\NN. \label{eq24}
	\end{eqnarray}
	Note that
	\begin{eqnarray}
		((u''_n, u'_n-u')) & = & \int_0^b\langle u''_n, u'_n-u'\rangle dt \nonumber \\
		& = & \int_0^b\langle u''_n - u'', u'_n - u'\rangle dt + ((u'', u'_n-u')) \nonumber \\
		& = & \int_0^b\frac{1}{2}\frac{d}{dt}|u'_n - u'|^2 dt + ((u'', u'_n-u'))\ \mbox{(see Proposition \ref{prop2})} \nonumber \\
		& = & \frac{1}{2}|u'_n(b) - u'(b)|^2 + ((u'', u'_n-u')) \nonumber \\
		& & \mbox{(since\ $u'_n(0) = u'(0) = u_{1}$\ for all\ $n\in\NN$,\ see \eqref{eq22})} \nonumber \\
		& \Rightarrow & \liminf_{n\rightarrow\infty}((u''_n, u'_n-u')) = \frac{1}{2}\liminf_{n\rightarrow\infty} |u'_n(b) - u'(b)|^2 \geq 0. \label{eq25}
	\end{eqnarray}
	Also we have
	\begin{eqnarray}
		& & ((\hat{B}(u_n-u), u'_n-u')) = \int_0^b\frac{1}{2}\frac{d}{dt}\langle B(u_n -u), u_n-u\rangle dt \nonumber \\
		& & \frac{1}{2}\langle B(u_n-u)(b), (u_n-u)(b)\rangle \geq 0\ \mbox{(see hypothesis $H(B)$)} \nonumber \\
		& \Rightarrow & ((\hat{B}u, u'_n-u')) \leq ((\hat{B}u_n, u'_n-u'))\ \mbox{for all}\ n\in\NN. \label{eq26}
	\end{eqnarray}
	Recall that
	$$\epsilon_n^{\frac{1}{2}}||u_n||_{L^r(T, X)} \leq M_0\ \mbox{for all}\ n\in\NN\ \mbox{all}\ r\geq p\ \mbox{(see Proposition \ref{prop5})}.$$
	Suppose that $r_m\rightarrow+\infty$, $r_m\geq p$ for all $m\in\NN$. Then for every $n\in\NN$, $\epsilon_n^{\frac{1}{r_m}}\rightarrow1$ as $m\rightarrow\infty$. Invoking Problem 1.175 of Gasinski and Papageorgiou \cite{5}, we can find $\{m_n\}_{n\geq1}$ with $ m_n\rightarrow+\infty$ such that
	$$\epsilon_n^{\frac{1}{r_{m_n}}}\rightarrow1\ \mbox{as}\ n\rightarrow\infty.$$
	Therefore there exists $n_0\in\NN$ such that
	\begin{eqnarray*}
		& & \frac{1}{2} \leq \epsilon_n^{\frac{1}{r_{m_n}}}\ \mbox{for all}\ n\geq n_0, \\
		&&\frac{1}{2}||u'_n||_{L^{r_{m_n}}(T, X)}\leq M_0\ \mbox{for all}\ n\geq n_0,\\
		& \Rightarrow & ||u'_n||_{L^p(T, X)}\leq 2M_0\ \mbox{for all}\ n\geq n_0\ \mbox{(recall that $r_{m_n}\geq p$)}.
	\end{eqnarray*}
	On account of (\ref{eq15}) and since $y=u'$, we have
	\begin{equation}\label{eq27}
		u'_n\xrightarrow{w}u'\ \mbox{in}\ L^{p}(T, X).
	\end{equation}
	Then from (\ref{eq26}) and (\ref{eq27}) it follows that
	\begin{equation}\label{eq28}
		0\leq \liminf_{n\rightarrow\infty}((\hat{B}u_n, u'_n-u')).
	\end{equation}
	In addition, we have
	\begin{equation}\label{eq29}
		\epsilon_n\hat{K}_p(u'_n)\rightarrow0\ \mbox{in}\ L^{p'}(T, X^*)\ \mbox{as}\ n\rightarrow\infty\ \mbox{(see (\ref{eq18}))}.
	\end{equation}
	By Proposition \ref{prop5} and (\ref{eq27}) it follows that
	\begin{eqnarray*}
		& & \{u'_n\}_{n\geq1} \subseteq W_p(0, b)\ \mbox{is bounded}, \\
		& \Rightarrow & \{u'_n\}_{n\geq1} \subseteq L^p(T, H)\ \mbox{is relatively compact (see (\ref{eq4}))}.
	\end{eqnarray*}
	Therefore we have
	\begin{eqnarray}
		& & u'_n\rightarrow u'\ \mbox{in}\ L^{p}(T, H)\ \mbox{(see (\ref{eq27}))}, \nonumber \\
		& \Rightarrow & \int_0^b(f_n, u'_n-u')dt\rightarrow 0\ \mbox{as}\ n\rightarrow\infty\ \mbox{(recall that $p\geq2$)}. \label{eq30}
	\end{eqnarray}

	If in (\ref{eq24}) we pass to the limit as $n\rightarrow\infty$ and use (\ref{eq25}), (\ref{eq28}), (\ref{eq29}), (\ref{eq30}), then
	$$\limsup_{n\rightarrow\infty}((a(u'_n),u'_n-u'))\leq0.$$
	Invoking Theorem 2.35 of Hu and Papageorgiou \cite[p. 41]{9}, we have
	\begin{equation}\label{eq31}
		a(u_n)\xrightarrow{w}a(u')\ \mbox{in}\ L^{p'}(T, X^*)\ \mbox{as}\ n\rightarrow\infty.
	\end{equation}
	In (\ref{eq24}) we pass to the limit as $n\rightarrow\infty$ and use (\ref{eq15}) (with $v=u''$) (\ref{eq27}), (\ref{eq29}), (\ref{eq31}). We obtain
	\begin{eqnarray*}
		& & u'' + a(u') + \hat{B}u=f,\ u(0)=u_0, u'(0)=u_1, f\in S^2_{F(\cdot, u(\cdot), u'(\cdot))}, \nonumber \\
		& \Rightarrow & u\in S(u_0, u_1)\neq\emptyset.
	\end{eqnarray*}
The proof is now complete.
\end{proof}

\subsection{An example}
We illustrate the main abstract result of this paper with a hyperbolic boundary value problem. Let $\Omega\subseteq\RR^N$ be a bounded domain. We consider the following boundary value problem
\begin{equation}\label{eq32}
	\left\{
		\begin{array}{ll}
			\frac{\partial^2u}{\partial t^2} - {\rm div}\,(a(t, z)|Du_t|^{p-2}Du_t) + \beta(z)u_t - \Delta u = f(t, z, u) + \gamma u_t\ \mbox{in}\ T\times\Omega, \\
			u|_{T\times\partial\Omega} = 0,\ u(0, z)=u_0(z),\ u_t(0, z)=u_1(z),
		\end{array}
	\right\}
\end{equation}
with $ u_t=\frac{\partial u}{\partial t}$, $2\leq p\leq \infty$, $\gamma>0$.

The forcing term $f(t, z, \cdot)$ need not to be continuous. So, following Chang \cite{1}, to deal with (\ref{eq32}), we replace it by a multivalued problem (partial differential inclusion), by filling in the gaps at the discontinuity points of $f(t, z, \cdot)$. So we define
$$f_l(t, z, x)=\liminf_{x'\rightarrow x} f(t, z, x')\ \mbox{and}\ f_u(t, z, x)=\limsup_{x'\rightarrow{x}}f(t, z, x').$$
Then we replace (\ref{eq32}) by the following partial differential inclusion
\begin{equation}\label{eq33}
	\left\{
		\begin{array}{ll}
			\frac{\partial^2u}{\partial t^2} -{\rm div}\,(a(t, z)|Du_t|^{p-2}Du_t) + \beta(z)u_t - \Delta u \in \left[f_l(t, z, u), f_u(t, z, u)\right] \mbox{in}\ T\times\Omega, \\
			u|_{T\times\partial\Omega} = 0,\ u(0, z)=u_0(z),\ u_t(0, z)=u_1(z).
		\end{array}
	\right\}
\end{equation}

Our hypotheses on the data of (\ref{eq33}) are the following:

\smallskip
$H(a):$ $a\in L^\infty(T\times\Omega), a(t, z)\geq0$ for almost all $(t, z)\in T\times\Omega$.

\smallskip
$H(\beta):$ $\beta\in L^\infty(\Omega)$, $\beta(z)\geq0$ for almost all $z\in\Omega$.

\smallskip
$H(f):$ $f:T\times \Omega\times\RR\rightarrow\RR$ is a function such that
\begin{itemize}
	\item [(i)] $f_l,\ f_u$ are superpositionally measurable (that is, for all $u:T\times\Omega\rightarrow\RR$ measurable, the functions $(t, z)\mapsto f_l(t, z, u(t, z)),\ f_u(t, z, u(t, z))$ are both measurable);
	\item [(ii)] there exists $a\in L^2(T\times\Omega)$ such that
	$$|f(t, z, x)|\leq a_2(t, z)(1+|x|)\ \mbox{for almost all}\ (t, z)\in T\times\Omega,\ \mbox{all}\ x\in\RR.$$
\end{itemize}

Let $X=W^{1,p}_0(\Omega)$, $H=L^2(\Omega)$ and $X^*=W^{-1,p'}(\Omega)$. Then $(X, H, X^*)$ is an evolution triple with $X\hookrightarrow H$ compactly (by the Sobolev embedding theorem).

Let $A:T\times X\rightarrow X^*$ be defined by
$$\langle A(t, u), h\rangle = \int_\Omega a(t, z) |Du|^{p-2}(Du, Dh)_{\RR^N}dz + \int_\Omega \beta(z)uhdz\ \mbox{for all}\ u, h\in W^{1,p}_0(\Omega).$$
Then $A(t, u)$ is measurable in $t\in T$, continuous and monotone in $u\in W^{1,p}_0(\Omega)$ (hence, maximal monotone) and $\langle A(t, u), u\rangle\geq0$ for almost all $t\in T$, all $u\in W^{1, p}_0(\Omega)$.

Let $B\in\mathscr{L}(X, X^*)$ be defined by
$$\langle Bu, h\rangle = \int_\Omega(Du, Dh)_{\RR^N}dz\ \mbox{for all}\ u, h\in W^{1, p}_0(\Omega).$$
Clearly, $B$ satisfies hypothesis $H(B)$.

Finally, let $G(t, z, x)=\left[f_l(t, z, x), f_u(t, z, x)\right]$ and set
$$F(t, u, v) = S^2_{G(t, \cdot, u(\cdot))} + \gamma v\ \mbox{for all}\ u, v\in L^2(\Omega).$$

Hypothesis $H(f)$ implies that $F$ satisfies $H(F)$.

Using $A(t, u), Bu$ and $F(t, u, v)$ as defined above, we can rewrite problem \eqref{eq33} as the equivalent second order nonlinear evolution inclusion (1). Assuming that $u_0\in W^{1, p}_0(\Omega)$ and that $u_1\in L^2(\Omega)$, we can use Theorem \ref{th6} and infer that problem (\ref{eq30}) has a solution $u\in C^1(T, L^2(\Omega)) \cap C(T, W^{1,p}(\Omega))$ with $\frac{\partial u}{\partial t}\in L^p(\Omega, W^{1,p}_0(\Omega))$ and $\frac{\partial^2 u}{\partial t}\in L^{p'}(\Omega, W^{-1,p'}(\Omega))$.

Note that if $a=0, f(t, z, x)=x$ and $\gamma=0$, then we have the Klein-Gordon equation. If $f(t, z, x)=f(x)=\eta \sin x$ with $\eta>0$, then we have the sine Gordon equation.

\medskip
{\bf Acknowledgments}. This research  was supported in part by  the  Slovenian  Research  Agency
grants P1-0292, J1-7025, J1-8131, and N1-0064. V.D.~R\u adulescu acknowledges the support through a grant  of the Romanian Ministry
of Research and Innovation, CNCS--UEFISCDI, project number PN-III-P4-ID-PCE-2016-0130,
within PNCDI III.

\end{document}